\documentclass{amsart}[10pt]
\usepackage{amsfonts,amssymb,amsmath,amsthm}
\usepackage{url}
\usepackage{enumerate}
\usepackage{multicol,pifont}
\newtheorem{thm}{Theorem}[section]
\newtheorem{cor}[thm]{Corollary}
\newtheorem{lem}[thm]{Lemma}
\newtheorem{prop}[thm]{Proposition}
\theoremstyle{definition}

\theoremstyle{remark}
\newtheorem{rem}[thm]{\bf Remark}

\numberwithin{equation}{section}

\begin{document}
\title[ $p$-chaoticity and regular action of abelian $C^{1}$-diffeomorphisms  groups of $\mathbb{C}^{n}$ fixing a point]{ $p$-chaoticity and regular action
\\ of abelian $C^{1}$-diffeomorphisms  groups of $\mathbb{C}^{n}$ fixing a point}

\author{Yahya N'dao and Adlene Ayadi}

 \address{Yahya N'dao, University of Moncton, Department of mathematics and statistics, Canada}
 \email{yahiandao@yahoo.fr }
\address{Adlene Ayadi, University of Gafsa, Faculty of sciences, Department of Mathematics,Gafsa, Tunisia.}

\email{adlenesoo@yahoo.com}

\thanks{This work is supported by the research unit: syst\`emes dynamiques et combinatoire:
99UR15-15} 

\begin{abstract}

In this paper, we introduce the notion of regular action of any
subgroup $G$ of $Diff^{1}(\mathbb{C}^{n})$ on $\mathbb{C}^{n}$
(i.e. the closure of every orbit of $G$ in some open set is a
topological sub-manifold of $\mathbb{C}^{n}$).  We prove that the
action of $G$, can not be $p$-chaotic for every $0\leq p \leq
n-1$. $($i.e. If $G$ has a dense orbit then the set of all regular
orbit with order $p$ can not be dense in $\mathbb{C}^{n})$.
Moreover, we prove that the action of any abelian lie subgroup of
$Diff^{1}(\mathbb{C}^{n})$, \ , is regular.
\end{abstract}
\maketitle

\begin{multicols}{2}

\section{\bf Introduction }

Denote by $Diff^{1}(\mathbb{C}^{n})$ the group of all
$C^{1}$-diffemorphisms of $\mathbb{C}^{n}$.  Let $G$ be an abelian
subgroup of $Diff^{1}(\mathbb{C}^{n})$  such that $0\in Fix(G)$
and $dim(vect(L_{G}))=n$, where $vect(L_{G})$ is the vector space
generated by $L_{G}=\{Df(0),\ f\in G\}$ and $Fix(G)=\{x\in
\mathbb{C}^{n}:\ f(x)=x,\ \forall f\in G\}$ be the global fixed
point set of $G$.  We can assume that $0\in Fix(G)$, leaving to
replace $G$ by $T_{a}\circ G \circ T_{-a}$ for any translation
$T_{a}$ of any vector $a\in Fix(G)$. There is a natural action
$G\times \mathbb{C}^{n} \longrightarrow \mathbb{C}^{n}$. $(f, x)
\longmapsto f(x)$. For a point $x\in\mathbb{C}^{n}$, denote by
$G(x) =\{f(x), \ f\in G\}\subset \mathbb{C}^{n}$ the orbit of $G$
through $x$. Denote by $\overline{E}$ (resp. $\overset{\circ}{E}$
) the closure (resp. interior) of $E$. A topological space $X$ is
called a topological manifold with dimension $r\geq 0$ over
$\mathbb{C}$ if every point has a neighborhood homeomorphic to
$\mathbb{C}^{r}$. This means that the image of any topological
manifold by a homeomorphism is a topological manifold with the
same dimension. An orbit $\gamma$ is called \emph{regular} with
order $ord(\gamma)=m$ if for every $y\in \gamma$ there exists an
neighborhood $O$ of $y$ such that $\overline{\gamma}\cap O$ is a
topological sub-manifold of $\mathbb{C}^{n}$ with dimension $m$
over $\mathbb{C}$. In particular, $\gamma$ is locally dense in
$\mathbb{C}^{n}$ if and only if $m = n$, and it is discrete if and
only if $m = 0$. Notice that, the closure of a regular orbit is
not necessary a manifold. We say that the action of $G$ is
\emph{regular} on $\mathbb{C}^{n}$ if every orbit of $G$ is
regular.  The action of $G$ is called \emph{chaotic} if $G$ has a
dense orbit and the union of all periodic orbits is dense in
$\mathbb{C}^{n}$ (cf. \cite{RLD}, \cite{SS}, \cite{GCAKAN}). We
give a generalization of the chaos as follow: The action of $G$ is
called \emph{$p$-chaotic},  $0\leq p \leq n-1$, if $G$ has a dense
orbit and the union of all orbits with order $p$ is dense in
$\mathbb{C}^{n}$. See that every chaotic action is $0$-chaotic.
Here, the question to investigate is the following:\
\medskip

 {\it The natural action of any subgroup of $Diff(\mathbb{C}^{n})$ can be $p$-chaotic, $0\leq p \leq n-1$?}
\
\\

{\it The action of any abelian lie subgroup of
$Diff(\mathbb{C}^{n})$ can be regular?}
\medskip

The notion of regular orbit is a generalization of non exceptional
orbit defined for the action of any group of diffeomorphisms on
$\mathbb{C}^{n}$. A nonempty  subset $E \subset \mathbb{C}^{n}$ is
a minimal set if for every $y\in E$ the orbit of $y$ is dense in
$E$. An orbit with its closure is a \emph{Cantor} set is called an
exceptional orbit. Their dynamics were recently initiated for some
classes in different point of view, (see for instance,
[3],[4],[5],[6],[7],[9]).
\smallskip

The action of $G$ on $\mathbb{C}^{n}$ is said \emph{proper} if and
only if the pre-image of any compact set by the action map, is
compact (i.e. for every two compact subsets $K_{1}$ and $K_{2}$ of
$\mathbb{C}^{n}$, the subset $\{f\in G,\ \ f(K_{1})\cap
K_{2}\neq\emptyset\}$ of G is compact). It is well known, that if
the action of a lie group on $\mathbb{C}^{n}$ is proper then all
the orbits are  embedded submanifolds in $\mathbb{C}^{n}$ (see for
instance ~\cite{JJ} and ~\cite{VVY}). Remark that, a proper action
of any lie group is regular, this means that the regular action is
a generalization of the proper action. \ \\  In \cite{ACNPS}, A.C.
Naolekar and P. Sankaran construct chaotic actions of certain
finitely generated abelian groups on even-dimensional spheres, and
of finite index subgroups of $SL(n.\mathbb{Z})$ on tori. They also
study chaotic group actions via compactly supported homeomorphisms
on open manifolds.
\smallskip

In \cite{PWCV}, P.W.Michor and C.Vizman proved that some groups of
diffeomorphisms of a manifold $M$ act n-transitively for each
finite n (i.e. for any two ordred sets of $n$ different points
$(x_{1},\dots, x_{n})$ and $(y_{1},\dots, y_{n})$ in $M$ there is
a smooth diffeomorphism $f$ in the group such that
$f(x_{i})=y_{i}$ for each $i$) .
\smallskip

 In \cite{YaHaA}, the authors studied the minimality  of any abelian diffeomorphisms groups acting
 on $\mathbb{C}^{n}$ fixing a point and $dim(L_{G})=n$, whose generalize the structure's theorem given in
\cite{aAhM05} for abelian linear group. This paper can be viewed
as a continuation of these works.
\
\\

Our principal results can be stated as follows:
\smallskip

\begin{thm}\label{T:01} Let $G$ be an abelian subgroup of $Diff^{1}(\mathbb{C}^{n})$ such that $0\in Fix(G)$
and $dim(vect(L_{G}))=n$. If $G$ has a dense orbit then the set of all dense orbit is a
$G$-invariant open set, dense in $\mathbb{C}^{n}$.
\end{thm}
\smallskip

\begin{cor}\label{C:01} The natural action of any abelian subgroup of $Diff^{1}(\mathbb{C}^{n})$
such that $0\in Fix(G)$ and $dim(vect(L_{G}))=n$, can not be
$p$-chaotic for every $0\leq p\leq n-1$. In particular, it can not
be chaotic.
\end{cor}
\smallskip

\begin{thm}\label{T:2} The natural action of  any abelian lie subgroup of $Diff^{1}(\mathbb{C}^{n})$
 on $\mathbb{C}^{n}$ is regular.
\end{thm}
\smallskip

As a directly consequence of Theorems ~\ref{T:2} and ~\ref{T:01},
we prove the regularity action of any abelian linear group on
$\mathbb{C}^{n}$.
\smallskip

\begin{cor}\label{C:2} The natural action of any abelian subgroup of $GL(n, \mathbb{C})$  on $\mathbb{C}^{n}$ is regular and not chaotic.
\end{cor}
\smallskip

\section{{\bf Proof Theorem ~\ref{T:2} and corollary ~\ref{C:01} }} We will cite the
definition of the exponential map given in ~\cite{IL}.

 Denote by:\
\\
- $\mathrm{g}$ be the lie algebra associated to $G$.\\ - The
exponential map $exp: \mathrm{g}\longrightarrow G$ is defined in
above.\smallskip

\begin{lem}\label{La:00} Let $x\in \mathbb{C}^{n}$. Then $G(x)$ is regular with order $r\geq 0$ if and only if there exist
an open set $O_{x}$ containing $G(x)$ such that
$\overline{G(x)}\cap O_{x}$ is a manifold with dimension $r\geq
0$.
\end{lem}
\smallskip

\begin{proof} The directly sens is obvious by definition.
Conversely, let $x\in G(u)$ and $O$ be an open neighborhood of $x$
such that
 $\overline{G(u)}\cap O_{x}$ is a manifold with dimension $r\geq 0$ over $\mathbb{C}$. Let $y\in
G(u)$, then $y=f(x)$ for some $f\in G$. So $O'_{x}=f(O_{x})$ is an
neighborhood of $y$ and satisfying $\overline{G(u)}\cap
O'_{x}=f(\overline{G(u)}\cap O)$ is a manifold with dimension
$r\geq 0$ over $\mathbb{C}$. It follows that $G(u)$ is regular
with order $r$.
\end{proof}
\smallskip

\subsection{{\it Whitney Topology on
$C^{0}(\mathbb{C}^{n},\mathbb{C}^{n})$}}We will use the definition
of Whitney topology given in ~\cite{WDM}.
 For each open subset $U\subset \mathbb{C}^{n}\times \mathbb{C}^{n}$ let $\widetilde{U}\subset \mathcal{C}^{0}(\mathbb{C}^{n}, \mathbb{C}^{n})$ be the
set of continuous functions $g$, whose graphs
$\{(x,g(x))\in\mathbb{C}^{n}\times \mathbb{C}^{n},\ \ x\in
\mathbb{C}^{n}\}$ is contained in $U$. \smallskip We want to
construct a neighborhood basis of each function
$f\in\mathcal{C}^{0}(\mathbb{C}^{n}, \mathbb{C}^{n})$. Let
$K_{j}=\{x\in \mathbb{C}^{n},\ \ \|x\|\leq j\}$ be a countable
family of compact sets (closed balls with center 0) covering
$\mathbb{C}^{n}$ such that $K_{j}$ is contained in the interior of
$K_{j+1}$. Consider then the compact subsets $L_{j}
=K_{j}\backslash \overset{\circ}{\overbrace{K_{j-1}}}$, which are
compact sets, too. Let $\epsilon= (\varepsilon_{j})_{j}$ be a
sequence of positive numbers and then define $V_{(f;\epsilon)}
=\{f\in\mathcal{C}^{0}(\mathbb{C}^{n}, \mathbb{C}^{n})\ :\
\|f(x)-g(x)\|< \varepsilon_{j},\ \mathrm{for\ any} \ x\in L_{j},\
\forall j\}.$ We claim this is a neighborhood system of the
function $f$ in $\mathcal{C}^{0}(\mathbb{C}^{n}, \mathbb{C}^{n})$.
Since $L_{i}$ is compact, the set
$U=\{(x,y)\in\mathbb{C}^{n}\times\mathbb{C}^{n}\ :\ \|f(x)-y\|<
\varepsilon_{j},\ if\  x\in L_{j}\}$ is open. Thus,
$V_{(f;\epsilon)}= \widetilde{U}$ is an open neighborhood of $f$.
On the other hand, if $O$ is an open subset of
$\mathbb{C}^{n}\times\mathbb{C}^{n}$ which contains the graph of
$f$, then since $L_{j}$ is compact, it follows that there exists
$\varepsilon_{j}>0$ such that if $x\in L_{j}$ and $\|y-f(x)\|<
\varepsilon_{j}$, then $(x;y)\in O$. Thus, taking
$\widetilde{\epsilon}= (\varepsilon_{j})_{j}$ we have
$V_{(f;\widetilde{\epsilon})}\subset \widetilde{O}$, so we have
obtained the family $V_{(f;\epsilon)}$ is a neighborhood system of
$f$. Moreover, for each given $\epsilon= (\varepsilon_{j})_{j}$,
we can find a $C ^{\infty}$-function $\epsilon:
\mathbb{C}^{n}\longrightarrow\mathbb{R}_{+}$, such that
$\epsilon(x)< \varepsilon_{j}$ for any $x\in L_{j}$. It follows
that the family
$V_{(f;\epsilon)}=\{g\in\mathcal{C}^{0}(\mathbb{C}^{n},
\mathbb{C}^{n})\ :\ \|f(x)-g(x)\|< \epsilon(x),\ \mathrm{for\
every}\ x\in \mathbb{C}^{n}\}$ is also a neighborhood system.
\smallskip

\subsection{{\it Linear map}} For a subset $E\subset\mathbb{C}^{n}$,  denote by $vect(E)$ the vector subspace
of $\mathbb{C}^{n}$ generated by all elements of $E$. $E$  is
called $G$-invariant if $f(E)\subset E$ for any $f\in G$; that is
$E$ is a union of orbits. Set $\mathcal{A}(G)$ be the algebra
generated by $G$. For a fixed vector $x\in
\mathbb{C}^{n}\backslash\{0\}$, denote by:
\\
- $\Phi_{x}: \mathcal{A}(G) \longrightarrow
\Phi_{x}(\mathcal{A}(G))\subset\mathbb{C}^{n}$ the linear map
given by $\Phi_{x}(f)=f(x)$.\ \\ - $E(x)
=\Phi_{x}(\mathcal{A}(G))$.\
\smallskip

\begin{lem}\label{L:aaaaa101001} The linear map
$\Phi_{x}:\mathcal{A}(G)\longrightarrow E(x)$ is continuous.
\end{lem}
\smallskip

\begin{proof} Firstly, we take the restriction of the Whitney
topology to $\mathcal{A}(G)$. Secondly, let $f\in \mathcal{A}(G)$
and $\varepsilon>0$. Then for $\epsilon=(\varepsilon_{j})_{j}$
with $\varepsilon_{j}=\varepsilon$ and for $V_{(f;\epsilon)}$ be a
neighborhood system of $f$, we obtain: for every $g\in
V_{(f;\epsilon)}\cap \mathcal{A}(G)$ and for every $y\in L_{j}$,
$\|f(y)-g(y)\|<\varepsilon$, $\forall j$. In particular for
$j=j_{0}$ in which $x\in L_{j_{0}}$, we have
$\|f(x)-g(x)\|<\varepsilon$, so
$\|\Phi_{x}(f)-\Phi_{x}(g)\|<\varepsilon$. It follows that
$\Phi_{x}$ is continuous.
\end{proof}
\smallskip

\
\\
Denote by:\ \\ - $r(x)=dim(E(x))$.\ \\ - $U_{j}=\{y\in
\mathbb{C}^{n},\ \ r(y)\geq j\}$.
\begin{prop}\label{p:100} Let $G$ be a subgroup of
 $Diff^{1}(\mathbb{C}^{n})$. Suppose that $G$ has a
 dense orbit. Then $U_{n}$ is a $G$-invariant open subset of $\mathbb{C}^{n}$.
\end{prop}
\
\\

\begin{proof} Let $x\in \mathbb{C}^{n}$ such that $\overline{G(x)}=\mathbb{C}^{n}$,  then $x\in U_{n}$ and so $U_{n}\neq\emptyset$.
Let $y\in U_{n}$, then $E(y)=\mathbb{C}^{n}$ and so, there exist
$f_{1},\dots, f_{n}\in F_{y}$ such that the $n$ vectors
$f_{1}(y),\dots, f_{n}(y)$ are linearly independent in
$\mathbb{C}^{n}$. For all $z\in \mathbb{C}^{n}$, we consider the
Gram's determinant $$\Delta(z)=det\left(\langle f_{i}(z)\ |\
f_{j}(z)\rangle\right)_{1\leq i,j\leq n}$$ of the vectors
$f_{1}(z),\dots, f_{n}(z)$ where $\langle. | . \rangle$ denotes
the scalar product in $\mathbb{C}^{n}$. It is well known that
these vectors are independent if and only if $\Delta(z) \neq 0$,
in particular $\Delta(y)\neq 0$. Let $$V_{y}=\left\{z\in
\mathbb{C}^{n},\ \ \Delta(z) \neq 0\right\}$$ The set $V_{y}$
  is open in $\mathbb{C}^{n}$, because the map $z\longmapsto \Delta(z)$ is continuous. Now $\Delta(y)\neq 0$, and so
   $y\in V_{y}\subset U_{n}$.\
     The proof is completed.
\end{proof}
\medskip

The construction of the open $U$ given in \cite{YaHaA}, is the
same of $U_{n}$ if $G$ has a dense orbit.\
\\
 \begin{lem}\label{C1212:4}$($\cite{YaHaA}, Corollary 1.2$)$ Let $G$ be an abelian subgroup of  $Diff^{1}(\mathbb{C}^{n})$,
such that $0\in Fix(G)$ and $dim(vect(L_{G}))=n$. If $G$ has a
dense orbit then every orbit in $U_{n}$ is dense in
$\mathbb{C}^{n}$.
\end{lem}
\medskip

\begin{proof}[Proof of Theorem~\ref{T:01}] Suppose that the group $G$ has a dense orbit denoted
by $G(x)$, $x\in \mathbb{C}^{n}$. Let $\mathcal{L}$ be the set of
all dense orbits, so $\mathcal{L}\neq\emptyset$ since $x\in
\mathcal{L}$. let $y\in \mathcal{L}$ then $y\in U_{n}$ and
$U_{n}\subset \overline{G(y)}$. By Lemma~\ref{C1212:4}, for each
$z\in  U_{n}$, one has $\overline{G(z)}=\mathbb{C}^{n}$, so $y\in
U_{n}\subset \mathcal{L}$. By Proposition~\ref{p:100}, $U_{n}$ is
a non empty open subset of $\mathbb{C}^{n}$. This completes the
proof.
\end{proof}
\medskip

\begin{rem} By the proof of Theorem ~\ref{T:01},
$U_{n}=\mathcal{L}$.
\end{rem}
\medskip

\begin{proof}[Proof of Corollary~\ref{C:01}] Suppose that the action
of the group $G$ is $p$-chaotic, then $G$ has a dense orbit
denoted by $G(x)$, $x\in \mathbb{C}^{n}$. By Theorem~\ref{T:01},
the set $\mathcal{L}$ of all dense orbit is a dense open set in
$\mathbb{C}^{n}$. This means that if $\mathcal{P}$ is the union of
all regular orbits with order $p$, then $\mathcal{L}\cap
\mathcal{P}=\emptyset$, so $\mathcal{P}$ can not be dense in
$\mathbb{C}^{n}$. The proof is completed.
\end{proof}
\smallskip

\section{{\bf Regular action of abelian lie subgroups of $Diff^{1}(\mathbb{C}^{n})$}} \
We will cite the definition of the exponential map given in
~\cite{IL}.
\subsection{{\it  Exponential map }} In this section, we illustrate
the theory developed of the group $Diff(\mathbb{C}^{n})$ of
diffeomorphisms of $\mathbb{C}^{n}$. For simplicity, throughout
this section we only consider the case of $\mathbb{C}=\mathbb{R}$;
however, all results also hold for complexes case. The group
$Diff(\mathbb{R}^{n})$ is not a Lie group (it is
infinite-dimensional), but in many way it is similar to Lie
groups. For example, it easy to define what a smooth map from some
Lie group $G$ to $Diff(\mathbb{R}^{n})$ is: it is the same as an
action of $G$ on $\mathbb{R}^{n}$ by diffeomorphisms. Ignoring the
technical problem with infinite-dimensionality for now, let us try
to see what is the natural analog of the Lie algebra $\mathrm{g}$
for the group $G$. It should be the tangent space at the identity;
thus, its elements are derivatives of one-parameter families of
diffeomorphisms.\ \\ Let $\varphi^{t}: G\longrightarrow G$ be
one-parameter family of diffeomorphisms. Then, for every point
$a\in G$, $\varphi^{t}(a)$ is a curve in $G$ and thus
$\frac{\partial}{\partial t}\varphi^{t}(a)_{/t=0}=\xi(a)\in
T_{a}G$ is a tangent vector to $G$ at $a$. In other words,
$\frac{\partial}{\partial t}\varphi^{t}$ is a vector field on
$G$.\smallskip

\ \\ The exponential map $exp: \mathrm{g}\longrightarrow G$ is
defined by $exp(x)=\gamma_{x}(1)$ where $\gamma_{x}(t)$ is the
one-parameter subgroup with tangent vector at $1$ equal to $x$.\
\\ If $\xi\in\mathrm{g}$ is a vectorfield, then $exp(t\xi)$ should
be one-parameter family of diffeomorphisms whose derivative is
vector field $\xi$. So this is the solution of differential
equation $$\frac{\partial}{\partial
t}\varphi^{t}(a)_{/t=0}=\xi(a).$$ \ \\ In other words,
$\varphi^{t}$ is the time $t$ flow of the vector field. Thus, it
is natural to define the Lie algebra of $G$ to be the space
$\mathrm{g}$ of all smooth vector $\xi$ fields on $\mathbb{R}^{n}$
such that $exp(t\xi)\in G$ for every $t\in \mathbb{R}$.
\smallskip

 \begin{prop}\label{p:000+1} $($\cite{IL}, Theorem 3.29$)$ Let $G$ be a Lie group acting on  $\mathbb{C}^{n}$ with
 lie algebra $\mathrm{g}$  and let $x \in \mathbb{C}^{n}$.\
 \\
(i) The stabilizer $G_{x} = \{f \in G :\ \ f(x) = x\}$ is a closed
Lie subgroup in $G$, with Lie algebra $\mathfrak{h}_{x} = \{f \in
\mathrm{g}:\  f(x) = 0\}$. \
\\
(ii) The map $G_{/G_{x}}\longrightarrow \mathbb{C}^{n}$ given by
$f.G_{x}\longmapsto f(x)$ is an immersion. Thus,
 the orbit $G(u)$ is an immersed submanifold in $\mathbb{C}^{n}$.
 In particular
 $\mathrm{dim}(G(x))=\mathrm{dim}(\mathrm{g})-\mathrm{dim}(\mathfrak{h}_{x})$.
\end{prop}
\

Denote by  $p=dim(\mathrm{g})$. Since $G$ is abelian so is
$\mathrm{g}$. Set $f_{1},\dots, f_{p}\in \mathrm{g}$ be the
generators of $\mathrm{g}$.\ We let:
\\
- $exp: \mathrm{g}\longrightarrow G$ the lie exponential map
associated to $G$. \ \\ - $G_{0}$ be the connected component of
$G$ containing the identity map $id$. So $G_{0}$ is generated by
$exp(\mathrm{g})$ and it is an abelian lie subgroup of $G$. Since
$\mathrm{g}$ is abelian, $G_{0}=exp(\mathrm{g})$.\ \\ For a fixed
point $x\in \mathbb{C}^{n}$, denote by:\ \\ - $G_{x}=\{f\in
G_{0},\ \ f(x)=x \}$  the stabilizer of $G_{0}$ on the point $x$.
It is a lie subgroup of $G_{0}$.
\
\\
Denote by:\ \\ - $H$  be the algebra associated to $G_{x}$ and
$F_{x}$ is the supplement of $H_{x}$ in $\mathrm{g}$ (i.e.
$F_{x}\oplus H_{x}=\mathrm{g}$). By Proposition ~\ref{p:000+1}, we
have $H_{x}=\{f\in \mathrm{g},\ \ f(x)=0\}$ and
$$G_{0}=exp(F_{x})\circ exp(H_{x}).$$ In particular
$G_{0}(x)=\Phi_{x}(exp(F_{x}))$.\ \\ -
$V=\{exp(t_{1}f_{1}+\dots+t_{p}f_{p}),\ \ \ |t_{k}|<1\}$.
 \smallskip

\begin{prop}\label{p:1} Let $G$ be an abelian subgroup of $Diff^{1}(\mathbb{C}^{n})$, and $x\in \mathbb{C}^{n}$.
Then:\ \\ (i) $G_{0}(x)$ is the connected component of $G(x)$
containing $x$.\ \\ (ii) The restriction  $\Phi^{(1)}_{x}:
exp(F_{x})\cap V\longrightarrow \Phi_{x}(exp(F_{x})\cap V)\subset
G_{0}(x)$ of $\Phi_{x}$ to $exp(F_{x})\cap V$ is an homeomorphism.
\end{prop}
\medskip

\begin{proof} (i) By Lemma ~\ref{L:aaaaa101001}, the map $\Phi_{x}: \mathcal{A} (G)\longrightarrow E(x)\subset \mathbb{C}^{n}$ is a
continuous surjective linear map. The proof follows then from the
fact that $G_{0}(x)= \Phi_{x}(G_{0})$ and $G_{0}$ is connected.
\
\\ (ii) By Lemma ~\ref{L:aaaaa101001}, the map $\Phi^{(1)}_{x}$ is continuous, surjective.\ \\ It is injective: Indded, if $f,g\in exp(F_{x})\cap V$
such that $\Phi^{(1)}_{x}(f)=\Phi^{(1)}_{x}(g)$, then $f(x)=g(x)$,
so $g^{-1}\circ f(x)=x$. Hence $g^{-1}\circ f\in G_{x}\cap
exp(F_{x})=\{id\}$. It follows that $f=g$. \ \\
$(\Phi^{(1)}_{x})^{-1}:\Phi_{x}(exp(F_{x})\cap V)\longrightarrow
exp(F_{x})\cap V$ is continuous; inded, let $y=exp(f)(x)\in
\Phi_{x}(exp(F_{x})\cap V)$, $f\in F_{x}$ and $(y_{m})_{m}$ be a
sequence in $\Phi_{x}(exp(F_{x})\cap V)$  tending to $y$. Let
$(f_{1},\dots, f_{q})$ be a basis of $F_{x}$ and set
$y_{m}=exp(t_{1,m}f_{1}+\dots+t_{q,m}f_{q})(x)$ and
$y=exp(t_{1}f_{1}+\dots+t_{q}f_{q})(x)$, with $|t_{k}|< 1$ and
$|t_{k,m}|<1$. We can assume (leaving to take a subsequence) that
$\underset{m\to +\infty}{lim}t_{k,m}=s_{k}$, with $|s_{k}|\leq 1$
for every $k=1,\dots, q$. Write
$g=exp(s_{1}f_{1}+\dots+s_{q}f_{q})$ and
$g_{m}=exp(t_{1}f_{1}+\dots+t_{q}f_{q})$. By continuity of the
exponential map we have $(g_{m})_{m}$ tends to $g$ when $m\to
+\infty$. By continuity of $\Phi_{x}$ (Lemma ~\ref{L:aaaaa101001})
we obtain $y_{m}=\Phi_{x}(g_{m})$ tends to $y=\Phi_{x}(g)$, so
$s_{k}=t_{k}$ for every $k=1,\dots, p$. As
$g=(\Phi^{(1)}_{x})^{-1}(y)$ and
$g_{m}=(\Phi^{(1)}_{x})^{-1}(y_{m})$, it follows that
$(\Phi^{(1)}_{x})^{-1}(y_{m})$ tends to
$(\Phi^{(1)}_{x})^{-1}(y)$. This completes the proof.
\end{proof}
\smallskip

\subsection{{\it Wedge, Lie wedge and almost abelian notions}}
We will use the notion of wedge and Lie wedge given by K.H.
Hofmann in ~\cite{KHH} and ~\cite{KHH2}: \ \\ - A {\it wedge} or a
closed convex cone in a finite dimensional vector $\mathrm{g}$ is
a topologically subset $\omega$ with  $\omega+\omega=\omega$ and
$\lambda.\omega\subset \omega$ for every $\lambda\geq 0$. In
particular, any vector subspace of $\mathrm{g}$ is a wedge in
$\mathrm{g}$.\
\\ - $h(\omega)=(-\omega)\cap \omega$ is called the edge of the wedge.\ \\ - A
{\it Lie wedge} $\omega$ in a Lie algebra $\mathrm{g}$ is a wedge
such that $$exp(ad(x))\omega=\omega,\ \ \mathrm{for\ all}\ x\in
h(\omega).$$ In particular, any subalgebra of $\mathrm{g}$ is a
Lie wedge in $\mathrm{g}$.\ \\ -  A Lie algebra $\eta$ is called
{\it almost abelian} if there is a linear form
$\alpha:\eta\longrightarrow \mathbb{R}$ such that the bracket is
given by $$[X, Y]=\alpha(X)Y-\alpha(Y)X.$$ In particular, any
abelian Lie algebra is almost abelian for  $\alpha=0$. If
$\alpha\neq 0$ the $\eta$ is called {\it truly almost abelian}. \
\  \\
\begin{lem}~\label{L1001:001}$($~\cite{KHH2}, Theorem 4.3$)$ Let
$\mathrm{g}$ be a Lie algebra, then the following are equivalent:\
\\ (i) $\mathrm{g}$ is almost abelian. \ \\ (ii) Every wedge is a
Lie wedge. \ \\ (iii) For every Lie wedge $\omega$, we have
$\overline{\langle exp(\omega)\rangle}=exp(\omega),$ where
$\langle exp(\omega)\rangle$ is the group generated by
$exp(\omega)$.
\end{lem}
\smallskip

As a consequence of above Lemma, we obtain:
\

\begin{cor}~\label{C1001:001++} We have $exp(F_{x})$ is a lie
subgroup of $G_{0}$.
\end{cor}
\smallskip

\begin{proof} By definition, we have $F_{x}$ is a wedge and by
Lemma~\ref{L1001:001}, (ii), it is a Lie wedge because
$\mathrm{g}$ is almost abelian (since it is abelian). Then by
Lemma~\ref{L1001:001}, (iii), we have $\overline{\langle
exp(F_{x})\rangle}=exp(F_{x})$, so $exp(F_{x})$ is closed subgroup
of $G_{0}$. It follows that $F_{x}$ is a Lie group.
\end{proof}
\smallskip

\begin{cor}\label{CC:0+03} $($Under  notations of Proposition
~\ref{p:1}$)$ The set $B(x)=\Phi_{x}(exp(F_{x})\cap V)$ is a
topological submanifold of $\mathbb{C}^{n}$ containing $x$.
Moreover, there exists an open subset $W$ of $\mathbb{C}^{n}$ such
that $W\cap G(x)=B(x)$.
\end{cor}
\smallskip

\begin{proof} By Corollary ~\ref{C1001:001++}, $exp(F_{x})$ is a lie subgroup of
$G_{0}$, so it is a topological manifold. By Proposition
~\ref{p:1}, $B(x)$ is homoeomorphic to $exp(F_{x})\cap V$ wich is
an open subset of $exp(F_{x})$. Then $B(x)$ is a topological
manifold with dimension equal to dim$(exp(F_{x}))$. On the other
hand, by (i), $G_{0}(x)=\Phi_{x}(exp(F_{x}))$ is a connected
component of $G(x)$ containing $x$, then there exists an open
subset $O$ of $\mathbb{C}^{n}$ such that $O\cap G(x)=G_{0}(x)$.
Since the exponential map $exp$ is a locally diffeomorphism on a
neighborhood of $0$ then dim$(exp(F_{x}))=\mathrm{dim}(F_{x})$, so
dim$(B(x))=\mathrm{dim}(exp(F_{x}))=\mathrm{dim}(F_{x})$. By
Proposition ~\ref{p:000+1}, $G_{0}(x)$ is  an immersed submanifold
of $\mathbb{C}^{n}$ with dimension
dim$(F_{x})=\mathrm{dim}(\mathrm{g})-\mathrm{dim}(H_{x})$ because
$\mathrm{g}$ is also the lie algebra of $G_{0}$. Therefore
dim$(B(x))=\mathrm{dim}(G_{0}(x))$, so $B(x)$ is an open subset of
$G_{0}(x)$. Then there exists an open subset $W$ of
$\mathbb{C}^{n}$ containing $x$ and contained in $O$ such that
$G_{0}(x)\cap W=B(x)$. It follows that $W\cap G(x)=G_{0}(x)\cap
W=B(x)$. The proof is completed.
\end{proof}

\smallskip

\begin{lem}\label{L:alal}For every neighborhood  $W$ of a point $x\in \mathbb{C}^{n}$, we have $\overline{G(x)}\cap W=\overline{G(x)\cap W}\cap W$.
\end{lem}
\smallskip

\begin{proof} It is clear that $\overline{G(x)\cap W}\cap W\subset \overline{G(x)}\cap
W$. Now, let $y\in \overline{G(x)}\cap W$ then there exists a
sequence $(y_{m})_{m}$ in $G(x)$ tending to $y$. So $y_{m}\in W$
from some row $m_{0}$. Thus $y\in \overline{G(x)\cap W}\cap W$.
\end{proof}
\smallskip

\begin{proof}[Proof of Theorem~\ref{T:2}] Let $G$ be an abelian subgroup of $Diff^{1}(\mathbb{C}^{n})$. By Corollary ~\ref{CC:0+03}, there exists
an open subset $W$ of $\mathbb{C}^{n}$ such that $W\cap G(x)=B(x)$
is a submanifold of $\mathbb{C}^{n}$. So $B(x)$ is locally closed,
we can assume that $\overline{B(x)}\cap W=B(x)$. Therefore, by
Lemma ~\ref{L:alal} we have  $\overline{G(x)}\cap W=
\overline{G(x)\cap W}\cap W$, so $B(x)\subset\overline{G(x)}\cap
W= \overline{G(x)\cap W}\cap W =\overline{B(x)}\cap W=B(x)$. Hence
$\overline{G(x)}\cap W=B(x)$ is a topological manifold. By
Lemma~\ref{La:00}, it follows that $G(x)$ is regular. We conclude
that the action of $G$ is regular.
\end{proof}
\medskip

Let $M_{n}(\mathbb{C})$  be the set of all square matrix over
$\mathbb{C}$ with order $n$ and $GL(n, \mathbb{C})$ be the group
of all reversible matrix of $M_{n}(\mathbb{C})$. Let $L$ be an
abelian subgroup of $GL(n, \mathbb{C})$, denote by:\
\\ - $\widetilde{L}=\overline{L}\cap GL(n, \mathbb{C})$, where
$\overline{L}$ is the closure of $L$ in $M_{n}(\mathbb{C})$. It is
clear that $\widetilde{L}$ is a lie subgroup of $GL(n,
\mathbb{C})$. \ \\ - $\overline{L}(x)=\{Ax,\ \ A\in
\overline{L}\}$.\ \\ We will use the following lemma to prove
Corollary ~\ref{C:2}.
\smallskip

\begin{lem}\label{L:04} For every $x\in \mathbb{C}^{n}$. We have
$\overline{\widetilde{L}(x)}=\overline{L(x)}$.
\end{lem}
\smallskip

\begin{proof}We have $\overline{L(x)}\subset\overline{\widetilde{L}(x)}$. Let
$y\in \overline{\widetilde{L}(x)}$, so
$y=\underset{m\to+\infty}{lim}A_{m}(x)$ for some sequence
$(A_{m})_{m\in \mathbb{N}}$ in $\widetilde{G}$. Therefore, for
every $m\in\mathbb{N}$, there exists a sequence
 $(A_{m,k})_{k\in\mathbb{N}}$ in $G_{/E(x)}$ tending to
 $A_{m}$. Then  $\underset{k\to
 +\infty}{lim}A_{m,k}x=A_{m}x$. Thus for every $\varepsilon>0$, \ there exists $M>0$ and  for
every $m\geq M$, there exists $k_{m}>0$,
 such that for every $k\geq k_{m}$, we have $\|A_{m}x-y\|<\frac{\varepsilon}{2}$ and  $\|A_{m,k}x-A_{m}x\|<\frac{\varepsilon}{2}$.\ Then,
for every $m>M$, $$\|A_{m,k_{m}}x-y\|\leq \|A_{m,k_{m}}x-A_{m}x\|+
\|A_{m}x-y\|<\varepsilon.$$ Hence $\underset{m\to
+\infty}{lim}A_{m,k_{m}}x=y$, so $y\in \overline{G(x)}$. It
follows that
 $\overline{\widetilde{G}(u)}\subset \overline{G(u)}$. The proof is completed.
\end{proof}\
\smallskip

\begin{proof}[Proof of Corollary~\ref{C:2}] By Theorem~\ref{T:2}, the action of $\widetilde{L}$ is regular.
Then for each $x\in \mathbb{C}^{n}$, there exists an open subset
$O$ of $\mathbb{C}^{n}$ such that $\overline{\widetilde{L}(x)}\cap
O$ is a topological sub-manifold of $\mathbb{C}^{n}$. It follows
by Lemma ~\ref{L:04}, that $\overline{L(x)}\cap O$ is a
topological sub-manifold of $\mathbb{C}^{n}$. The proof is
completed.
\end{proof}
\smallskip

\bibliographystyle{amsplain}
\vskip 0,4 cm

 \end{multicols}

\end{document}